\documentclass[12pt]{amsart}

\usepackage{xypic,enumerate}

\usepackage{amsthm,amsmath,amssymb}

\usepackage{graphicx,  ctable, comment}

\usepackage[colorlinks=true,citecolor=black,linkcolor=black,urlcolor=blue]{hyperref}

 \usepackage{time, comment, hyperref, latexsym}
      \usepackage{amssymb, mathrsfs,  chngcntr, mathdots, tabularx,}

      \usepackage{stmaryrd,  fancyhdr, array, enumerate}
      
\theoremstyle{plain}
\newtheorem{theorem}{Theorem}[section]
\newtheorem{lemma}[theorem]{Lemma}
\newtheorem{corollary}[theorem]{Corollary}
\newtheorem{proposition}[theorem]{Proposition}
\newtheorem{fact}[theorem]{Fact}
\newtheorem{observation}[theorem]{Observation}

\theoremstyle{definition}
\newtheorem{definition}[theorem]{Definition}

\theoremstyle{remark}
\newtheorem{remark}[theorem]{Remark}

\newcommand{\fm}{\mathfrak{m}}
\newcommand{\fn}{\mathfrak{n}}
\newcommand{\fq}{\mathfrak{q}}
\newcommand{\fp}{\mathfrak{p}}
\newcommand{\FF}{\mathbb{F}}

\newcommand{\Nn}{\mathbb{N}}

\newcommand{\coker}{\operatorname{coker}}

\def\ext{\operatorname{Ext}}
\def\hm{\operatorname{Hom}}
\def\tr{\operatorname{Tr}}

\def\im{\operatorname{im}}

\def\pd{\operatorname{pd}}
\def\rank{\operatorname{rank}}
\def\length{\operatorname{length}}
\def\ann{\operatorname{ann}}
\def\gl{\operatorname{GL}}
\def\soc{\operatorname{soc}}

\numberwithin{equation}{section}

\begin{document}

\title[A Description of Totally Reflexive Modules]{A Description of Totally Reflexive Modules for a Class of \\
non-Gorenstein Rings}

\author[Denise A. Rangel Tracy]{Denise A. Rangel Tracy}

\address{ Department of Mathematics,  Syracuse University, Syracuse, NY USA}

\email{detracy@syr.edu}



\date{\today}



\begin{abstract}
We consider local non-Gorenstein rings of the form $(S_i,\fn_i)=k[X, Y_1, \ldots ,Y_i]/\left(X^2, (Y_1, \ldots, Y_i)^2\right), $ where $i\geq 2.$ We show that every totally reflexive $S_i$-module has a presentation matrix of the form $I x + \sum_{j=1}^i B_j y_j, $ where $I$ is the identity matrix and $B_j$ is an square matrix with entries in the residue field, $k$. From there, we prove that there exists a bijection between the set of isomorphism classes of totally reflexive modules (without projective summands) over $S_i$ which are minimal generated by $n$ elements and the set  of $i$-tuples of  $n \times n$ matrices with entries in $k$ modulo a certain equivalence relation. 

 \end{abstract}

\maketitle
\thispagestyle{empty}

\section{Introduction and Preliminaries}
For a two-sided Noetherian ring $R,$ a finitely generated $R$-module $M$ is called \emph{totally reflexive} if $\ext_R^i(M,R) =0, $ $\ext_R^i(\hm_R(M,R), R) =0 $  for all $i>0,$ and the biduality map, $\delta: M \rightarrow \hm_R(\hm_R(M,R), R)$ is an isomorphism. These modules were originally called modules of Gorenstein dimension zero \cite{AusBr} and are also known as finitely generated Gorenstein projective modules \cite{EdJen}. The term totally reflexive was introduced in 2002 \cite{AvMar}.

In this paper, we assume all rings to be commutative Noetherian and local. Projective modules, and equivalently free modules, are obviously totally reflexive. We call a nonzero totally reflexive module \emph{trivial} if it is a projective module.  A complex is called \emph {acyclic} if its homology is zero.  A \emph{totally acyclic complex} is an acyclic complex $\mathbb{A}$ whose \emph{dual} $\hm_R(\mathbb{A}, R)$ is also acyclic. Recall that over a local ring $(R, \fm)$ a complex $(\FF, \partial)$ of free modules is said to be \emph{minimal} if $\im \partial_i^{\FF} \subseteq \fm\FF_{i-1}$ for all $i.$

 A module $T$ being totally reflexive is equivalent to being a syzygy in a totally acyclic complex of finitely generated free modules. We call this totally acyclic complex a \emph{complete resolution} of $T.$  The ranks of the free modules in a complete (also free) resolution are called \emph{Betti numbers.}

Important examples of totally reflexive modules include all maximal Cohen-Macaulay modules over Gorenstein rings. The representation theory of maximal Cohen-Macaulay modules has been well documented, for example see \cite{LWbook}, \cite{Yosbook}.  Therefore, the class of rings we will be considering are non-Gorenstein. There is a special type of ring element called an exact zerodivisor, see section \ref{Describ} for details, which plays an important role in the study of totally reflexive modules over non-Gorenstein rings. The ideals generated by such elements are totally reflexive modules as well. 

 It is known from  \cite{sing} that if a non-Gorenstein local ring $(R,\fm,k)$ admits one  nontrivial totally reflexive module, then there exist infinitely many non-isomorphic indecomposable totally reflexive modules. In addition, if $R$ contains an exact zerodivisor and $k$ is algebraically closed, then for every $n\in \Nn$ there exists a family $\{M_n^{\lambda}\}_{\lambda\in k}$ of indecomposable, pairwise non-isomorphic totally reflexive $R$-modules, \cite[Theorem 1.4]{BT}. This is referred to as having wild totally reflexive representation type. 

In \cite{BT} and \cite{Holm}, infinite families of non-isomorphic indecomposable totally reflexive module are constructed, both of which arise from exact zerodivisors. This is done for  arbitrary local commutative rings. However, there exist examples of totally reflexive modules which cannot be obtained from these constructions. For example, the rings that are the focus of this paper.
For the following class of rings we are able to give a complete description of the isomorphism classes totally reflexive modules via presentation matrices. From this description, we can clearly see that these rings do in fact have wild totally reflexive representation type.

\begin{definition} Let $k$ be a field of characteristic zero and for $i \geq 2 $ set 
$$(S_i, \fn_i)=k[X, Y_1, \ldots ,Y_i]/\left(X^2, (Y_1, \ldots, Y_i)^2\right).$$
Set  $x=X+(X^2, (Y_1, \ldots, Y_i)^2) $   and $ y_j=Y_j +(X^2, (Y_1, \ldots, Y_i)^2) $ for $j=1, \ldots i.$ 

\end{definition}

 These can be easily identified as non-Gorenstein rings by computing the rank of the socle of $S_i$. That is, $\rank_k(\soc(S_i))=\rank_k(\ann(\fn_i))=\rank_k((xy_1,xy_2, \dots, xy_i))=i\geq 2.$  Let $\mbox{M}
_n(k)$ be the set of all $n \times n $ matrices with entries in $k$.  Within the class of rings $S_i$, we are able to prove the following.\\ 

\textbf{Theorem \ref{isomclass}. }\emph{For $i \geq 2, $ let $\mathcal{T\!\!  R}_n(S_i)$ be the set of isomorphism classes of totally reflexive modules  over $S_i,$ without projective summands, which is minimal generated by $n$ elements. 
 Then there exists a bijection, 
$$\vartheta: \mathcal{T\!\!  R}_n(S_i) \leftrightarrow \mbox{M}_n(k)^i/ \sim \, ,$$
where $(A_1, \dots, A_i) \sim (B_1, \dots, B_i)$ if there exists a  $P \in\mbox{GL}_n(k)$ such that \\
$P(A_1, \dots, A_i)P^{-1}:=(PA_1P^{-1}, \dots,  PA_i P^{-1})=(B_1, \dots, B_i).$}\\

This paper concludes with the proof of this theorem. If $F_1 \xrightarrow{\varphi} F_0 \to M\to 0$ is a minimal free presentation of a module $M,$ then any matrix representing the differential $\varphi$ is called a \emph{presentation matrix.} Theorem \ref{isomclass} follows from the fact, which will be proven later, that every totally reflexive $S_i$-module has a presentation matrix of the form 
$$I_n x + \sum_{j=1}^i B_j y_j, $$ where $B_j$ is an $n \times n$ matrix with entries in $k$, see Theorem \ref{presmatrixform}.
 

\section{Presentation Matrices of Totally Reflexive \\Modules over $S_i$}
  Presentation matrices are not unique. However, over commutative Noetherian local rings finitely generated isomorphic modules have equivalent minimal presentation matrices, see \cite[Theorem 4.3]{equivmat}. That is, if $\Phi$ and $\Psi$ are minimal presentation matrices of isomorphic modules, then there exist invertible matrices, $A$, $B$ such that $\Phi=A\Psi B^{-1}.$ In other words, minimal presentation matrices are unique up to elementary row and column operations. 

We have chosen the rings $S_i$ since they are a type of non-Gorenstein rings for which we know many facts about their totally reflexive modules, \cite{Acyclic},\cite {Yos}. Specifically,  they are non-Gorenstein local rings with the cube of the maximal ideal equal to zero. In addition, a result due to Yoshino classifies the structure of rings that admit nontrivial totally reflexive modules.

\begin{theorem} \cite[Theorem 3.1]{Yos} \label{YosThm}
Let $(R,\fm)$ be a commutative local Artinian non-Gorenstein ring with $\fm^3=0\neq\fm^2$ and let $e=\rank(\fm/\fm^2).$ If there exists a nontrivial totally reflexive $R$-module $T,$ then the following conditions hold:
\begin{enumerate}[(a)]
\item $\ann\fm=\fm^2$
\item $\length (\fm^2)=e-1,$ in particular, $\length(R)=2e.$ \label{YosThmlength}
\item  $\length (T)=ne,$ were $n$ is the minimal number of generators of $T,$ 
\item $R$ has a natural structure of homogeneous graded ring with \\$R=R_0\oplus R_1 \oplus R_2.$
\item The module $T$ has a natural structure of a  graded $R$-module. If $T$ has no free summands, then $T$ has only two graded pieces;\\ $T=T_0\oplus T_1$
\item \label{constbetti}$T$ has a minimal free graded resolution of the form \label{YosThmReso}
$$\cdots \rightarrow R^n(-n) \rightarrow \dots \xrightarrow{d_1} R^n(-2) \xrightarrow{d_2} R^n(-1) \xrightarrow{d_1}R^n \rightarrow T \rightarrow 0.$$
\end{enumerate}
\end{theorem}
There are a few of particularly important facts to note from this theorem stemming from  part (\ref{constbetti}). First,  a minimal complete resolution of a totally reflexive module over such a ring would have constant Betti numbers. Hence, any minimal presentation matrix of a totally reflexive module must be square. Moreover, every entry in this matrix would be of a linear form. That is, completely contained in $T_0$. This can be seen by noting the grading in the free resolution of $T$ decreases by one at each step. Hence, $T$ has a linear
resolution.

In order to continue with our analysis, we need to consider another class of rings. A surjection of local rings $\rho: (P, \fq, k) \to (R, \fm, k)$ is called a \emph{deformation} of $R$ if $\ker \rho$ is generated by a $P$-regular sequence. A deformation is called \emph{embedded} if $\ker \rho \subset \fq^2.$ This situation is often referred to as $P$ being an embedded deformation of $R$.  Finding an embedded deformation of $S_i$ will enable us to completely describe the presentation matrices of totally reflexive $S_i$-modules.

\begin{observation} The ring $(Q_i, \fp_i)=k[X, Y_1, \ldots ,Y_i]/\left( (Y_1, \ldots, Y_i)^2\right)$ is an embedded deformation of $S_i, $ and since $x^2$ is a  regular element of $Q_i,$ we have 
$S_i\cong Q_i/(x^2).$ \end{observation}
 For the reminder of this paper, we will reserve $(Q_i, \fp_i)$ to denote the ring defined above.

\begin{fact}\label{Gdimfact}  \cite[4.2]{Yos} \label{pd1} If a ring $P$ is an embedded deformation of a non-Gorenstein ring $R$ that admits  nontrivial totally reflexive module  $T$, then $\pd_P(T)=1.$ \end{fact}

Therefore, for any totally reflexive $S_i$-module $T$ 
$$0 \xrightarrow{} Q_i^n\xrightarrow{\partial} Q_i^n\xrightarrow{}  T  \xrightarrow{}0 $$
is a $Q_i$-free resolution of $T.$  


 \section{Description of  $\mathcal{T\!\!  R}_n(S_i)$}\label{Describ}
 The goal of this section is to consider a free presentation of a totally reflexive $S_i$-module $T,$
$$ S_i^{n} \xrightarrow{\partial} S_i^{n} \rightarrow T \rightarrow 0$$
and to investigate the possibilities for $\partial,$ as well as the other differentials in a complete resolution of $T$ and be able to give it a standard form.  

We begin by first investigating one of simplest examples of nontrivial totally reflexive modules, namely those given by exact zerodivisors.  For a commutative ring $R$, a non-unit $s \in R$ is said to be an \emph {exact zerodivisor} if there exists $t \in R$ such that $\ann (s)= tR$  and $\ann (t) = sR.$ If $R$ is local, then $t$ is unique up to units, and we call $s, t$ an \emph {exact pair of  zerodivisors}. This is equivalent to the existence of a totally acyclic complex 
$$ \cdots \rightarrow R \xrightarrow{t}R \xrightarrow{s}R \xrightarrow{t}R \xrightarrow{s}R \rightarrow \cdots.$$
Therefore, the modules $sR$ and $tR$ are totally reflexive.\\

For the rings $S_i,$ we know the form of the exact zerodivisors. 

\begin{proposition} Assume $i\geq 2. $ For $a,b_j \in k,$ the ring element $ax +\sum_{j=1}^i b_j y_j $ is an exact zerodivisor in $S_i$ if and only if $a\neq 0$.  In addition, all cyclic totally reflexive $S_i$-modules are obtained from exact zerodivisors.

\end{proposition}

\begin{proof} First suppose that $a=0, $ so that $s=\sum_{j=1}^i b_j y_j . $ Then we have that $\ann(s)= (y_1, y_2, \ldots, y_i). $ Since this is not principal, $s$ cannot be a exact zerodivisor.

 Assuming $a\neq 0,$  let $s=x +\sum_{j=1}^i c_j y_j $ and $t=x -\sum_{j=1}^i c_j y_j, $  where $c_j=\frac{b_j}{a}.$ Since $ts=0,$ we have $tR\subseteq \ann(s)$. Now note that the ring $S_i$ is also standardly graded, and that $\{x, y_j\}_{j=1,...,i}$ is a basis of its linear elements. 

Let $f \in \ann (s)\subset S_i,$ so $f= \alpha x +\sum_{j=1}^i \beta_j y_j +\sum_{j=1}^i \gamma_j x y_j,$ for some $\alpha,  \beta_j, \gamma_j \in k. $ Since $f$ annihilates $s$ we have  $0=fs=\alpha x \sum_{j=1}^i c_j y_j+x\sum_{j=1}^i \beta_j y_j, $ since $y_hy_j=0$ for all $h,j=1, \ldots i.$ Since the set  $\{y_j\}_{j=1,...,i}$ is part of a basis of the linear elements in $S_i,$ we have that  $\beta_j=-\alpha c_j$ for all $j=1, \ldots i.$ In addition, one can write $f$ as multiple of $t.$ That is, 
$$f= \alpha x-\sum_{j=1}^i \alpha c_j y_j +\sum_{j=1}^i \gamma_j x y_j= (\alpha + \sum_{j=1}^i \gamma_j  y_j) t \in tR.$$ Therefore, $\ann(s)=tR. $ Showing that $\ann(t)=sR$ is similar and thus omitted. Therefore, $s$ and $t$ form an exact pair of zerodivisors. \end{proof}


The form of the exact zerodivisors in $S_i$ foreshadows our form for presentation matrices of larger totally reflexive $S_i$-modules.  

\begin{theorem} \label{presmatrixform} If $T$ is a totally reflexive $S_i$-module which is minimally generated by $n$  elements, then there exists a presentation matrix of $T$ of the form $x I_n + \sum_{j=1}^i B_j y_j$ where $B_j$ is an $n\times n$ matrix with entries in $k$ for $j=1, \dots i.$  

\end{theorem}

\begin{proof}
 Let 
$$ \mathbf{T} =\left[\begin{array} {c c c}
a_{11} & \dots & a_{1n}\\
\vdots & \ddots & \vdots \\
a_{n 1} & \dots & a_{nn}\\
\end{array} \right]$$

\noindent be a minimal presentation matrix of a totally reflexive $S_i$-module $T$. By Lemma \ref{YosThm} part (\ref{constbetti}) we can assume that for all $1\leq g, h \leq n$ we have $a_{gh}= b_0x+\sum_{j=1}^i b_j y_j$ where $b_j \in k.$ From  the fact given by \ref{Gdimfact} the  pd$_{Q_i} (T) = 1,$  and a minimal $Q$-free resolution of $T$ is given by 
$$ 0 \to Q^n \xrightarrow{\partial} Q^n \to T \to 0.$$
Thus for $\mathbf{c}=\left[\begin{array} {ccc} c_1& \dots & c_n \end{array} \right]^T$ in  $ Q_i^n,$ \:if $\mathbf{Tc}=0,$ then $\mathbf{c}= 0.$ That is, if for all $1 \leq g\leq n$
$$c_1a_{g1} + c_2a_{g2} + \dots + c_n a_{gn} =0, 
$$
then $c_w=0,$ for all $1 \leq w\leq n.$ If $a_{gh}= \sum_{j=1}^i b_j y_j$ where $b_j \in k,$ then it would have the relations $a_{gh}\sum_{j=1}^i b_j y_j=0.$

Therefore, there must exists at least one value of $g$ for every $h$ such that $a_{gh}= \lambda x + \sum_{j=1}^i b_j y_j$ with $\lambda \neq 0.$  By the use of row and column operations, one can obtain a presentation matrix of $T$ in which there is exactly one entry in each row and in each column that contains an $x$ term, and without lost of generality, choose $a_{gg}$ to be that term. Therefore, there is a presentation matrix of $T$ of the form $x I_n + \sum_{j=1}^i B_j y_j$ where $B_j$  is  an $n\times n$ matrix with entries in $k$.
\end{proof}

\begin{lemma}\label{periodic} 
If $T$ is a totally reflexive $S_i$-module, then a complete resolution of $T$ is periodic of period 1 or 2. 

\end{lemma}

\begin{proof}

 Since $x^2$ annihilates $T$ as an $S_i$-module and $S_i\cong Q_i/(x^2),$ we can use the construction given by Avramov in  \cite[5.1.2]{infres}.  Since $x^2 \in Q_i$ is a regular element, multiplication by it is null-homotopic between the free $Q_i$-modules in a $Q_i$-resolution of $S_i$-modules.  More precisely, we have the following commutative diagram.
$$
\xymatrixrowsep{1.9pc}\xymatrixcolsep{2.3pc}\xymatrix{
&0 \ar@{->}[r]& \ar@{->}[d]_{x^2}Q_i^n\ar@{->}[r]^{\partial} & \ar@{->}[d]^{x^2}Q_i^n \ar@{->}[ld]_{\sigma}\ar@{->}[r] & T \ar@{->}[d]^{x^2} \ar@{->}[r]&0 \\
&0 \ar@{->}[r]& Q_i^n\ar@{->}[r]^{\partial} & Q_i^n \ar@{->}[r] & T \ar@{->}[r]&0 \\
}
$$

In particular, if we consider $\partial $ and $ \sigma $ as matrices with entries in $Q_i$, then they are a matrix factorization of $x^2.$ That is, $\partial \sigma = x^2I_n = \sigma\partial. $   Note that this holds regardless of  choice of a presentation matrix of $T.$ By the previous theorem, $T$ has a presentation matrix of the form $x I_n + \sum_{j=1}^i B_j y_j.$ Let the matrix $ \partial=x I_n + \sum_{j=1}^i B_j y_j. $ In order for $\partial \sigma = x^2I_n = \sigma\partial$ to hold, we must have  $ \sigma =x I_n - \sum_{j=1}^i B_j y_j.$ 

Define an infinite complex of free $S$-modules 
$$ \cdots\xrightarrow{S_i\otimes_{Q_i} \partial} S_i \otimes_{Q_i} Q_i^n\xrightarrow {S_i\otimes_{Q_i} \sigma} S_i \otimes_{Q_i} Q_i^n \xrightarrow{S_i\otimes_{Q_i} \partial} S_i \otimes_{Q_i} Q_i^n \xrightarrow{} \cdots,$$
which is isomorphic to 
$$\mathbb{A}: \qquad \cdots\xrightarrow{\overline{\partial}} S_i^n\xrightarrow {\overline{\sigma}} S_i^n \xrightarrow{\overline{\partial}} S_i^n \xrightarrow{}\cdots,$$
where $\overline{\partial}$ and $\overline{\sigma}$ are the matrices $\partial$ and $\sigma,$ but with entries views as $S_i$ elements.  We know that the complex $\mathbb{A}$ is minimal since we assumed  $T$ to be indecomposable, see \cite[5.1.2]{infres}.  
 \end{proof}


\begin{remark} 
The only indecomposable totally reflexive module whose differentials in a complete resolution are constant is $S_i/(x)$.

\end {remark}

With the use of Lemma \ref{periodic}, we can see what that differentials in a (minimal) complete resolution of a totally reflexive module have a certain form. 
\begin{corollary}
Using the language of Theorem \ref{presmatrixform}, if $T$ is a totally reflexive $S_i$-module, then 
$$ \cdots  \to S_i^n \xrightarrow{x I_n + \sum_{j=1}^i B_j y_j}S_i^n \xrightarrow{x I_n - \sum_{j=1}^i B_j y_j}S_i^n \xrightarrow{x I_n + \sum_{j=1}^i B_j y_j} S_i^n \to \cdots$$
is a complete resolution of $T$, where $T\cong \coker \left(x I_n + \sum_{j=1}^i B_j y_j\right)$.

\end{corollary}

We are now able to prove our main result.

\begin{theorem}  \label{isomclass} 
For $i \geq 2, $ let $\mathcal{T\!\!  R}_n(S_i)$ be the set of isomorphism classes of totally reflexive modules  over $S_i,$ without projective summands, which is minimal generated by $n$ elements. 
 Then there exists a bijection, 
$$\vartheta: \mathcal{T\!\!  R}_n(S_i) \leftrightarrow \mbox{M}_n(k)^i/ \sim \, ,$$
where $(A_1, \dots, A_i) \sim (B_1, \dots, B_i)$ if there exists a  $P \in\mbox{GL}_n(k)$ such that \\
$P(A_1, \dots, A_i)P^{-1}:=(PA_1P^{-1}, \dots,  PA_i P^{-1})=(B_1, \dots, B_i).$

\end{theorem} 

\begin{proof}  Define the map $\vartheta: \mathcal{T\!\!  R}_n(S_i) \leftrightarrow \mbox{M}_n(k)^i/ \sim \, $ by $\vartheta(\mathcal {B}) = (B_1, B_2, \ldots, B_i),$ where for any totally reflexive module $T$ in $\mathcal{B},$ $T$ has a presentation matrix of the form $x I_n + \sum_{j=1}^i B_j y_j.$ Let  $B, B' \in\mathcal {B},$ and suppose that they have the following presentation matrices 
$$B:\, x I_n + \sum_{j=1}^i B_j y_j \quad \mbox{ and } \quad B':\, x I_n + \sum_{j=1}^i B^{'}_{j} y_j.$$

Since $B$ and $ B'$ are in the same isomorphism class, there exist matrices $P,Q \in \gl_n(k)$ such that $P( xI_n + \sum_{j=1}^i B_j y_j)Q=x I_n + \sum_{j=1}^i B^{'}_{j} y_j.$ This implies that $PQ=I_n,$ and hence $Q=P^{-1}. $ Therefore, $(B_1, B_2, \ldots, B_i)$ is equivalent to $(B^{'}_1, B^{'}_2, \ldots, B^{'}_i).$ Lastly, it is easy to see that $\vartheta$ is both one-to-one and onto. 
\end{proof}


\end{document}